\documentclass[12pt]{amsart}
\usepackage[T1]{fontenc}
\usepackage[utf8]{inputenc}
\usepackage{fourier}
\usepackage[active]{srcltx}
\usepackage{a4wide}
\usepackage{amsthm,amsfonts,amsmath,mathrsfs,amssymb}
\usepackage{dsfont}
\def\be{\begin{equation}}
\def\ee{\end{equation}}

\newtheorem*{property*}{Property}
\newtheorem{claim}{Claim}
\newtheorem{problem}{Problem}
\newtheorem{conjecture}{Conjecture}
\newtheorem{remark}{Remark}
\newtheorem*{completeness*}{Completeness property}
\newtheorem*{theorem*}{Theorem}
\newtheorem{theorem}{Theorem}
\newtheorem{proposition}{Proposition}
\newtheorem*{proposition*}{Proposition}
\newtheorem{lemma}{Lemma}
\newtheorem{corollary}{Corollary}
\theoremstyle{remark}

\newcommand{\nc}{\newcommand}

\newcommand{\A}{2}
\newcommand{\rr}{r}
\newcommand{\kk}{N}

\newcommand{\F}{\mathbb{F}}
\newcommand{\PP}{\mathscr{P}}

\newcommand{\M}{\mathcal{M}}

\newcommand{\N}{{\mathbb N}}

\nc{\supp}{\operatorname{supp}}
\nc{\Real}{\operatorname{Re}}
\nc{\Imag}{\operatorname{Im}}
\nc{\dif}{\operatorname{d}} \nc{\im}{\operatorname{i}}

\nc{\Hi}{{\mathscr{H}}^\infty} \nc{\Ht}{{\mathscr{H}}^2}
\nc{\Hone}{{\mathscr{H}}^1} \nc{\ol}{\overline} \nc{\bz}{\mathbf{z}}
\nc{\bw}{\mathbf{w}} \nc{\eps}{\varepsilon}

\allowdisplaybreaks[3]

\begin{document}

\title[Extreme Values of Derivatives of  the Riemann zeta function]
{Extreme Values of Derivatives of  the Riemann zeta function}
\author{Daodao Yang}
\address{Institute of Analysis and Number Theory \\ Graz University of Technology \\ Kopernikusgasse 24/II, 
A-8010 Graz \\ Austria}

\email{yang@tugraz.at \quad yangdao2@126.com}

\maketitle

\begin{abstract}
It is proved  that  if $T$ is sufficiently large, then uniformly for all positive integers $\ell \leqslant (\log T) / (\log_2 T)$, we have
\begin{equation*}  
\max_{T\leqslant t\leqslant 2T}\left|\zeta^{(\ell)}\Big(1+it\Big)\right| \geqslant e^{\gamma}\cdot \ell^{\ell}\cdot (\ell+1)^{ -(\ell+1)}\cdot\Big(\log_2 T - \log_3 T + O(1)\Big)^{\ell+1} \,,
\end{equation*}
where $\gamma$ is the Euler constant. 
We also  establish lower bounds for maximum of $\big|\zeta^{(\ell)}(\sigma+it)\big|$ when $\ell \in \mathbb N $
 and $\sigma \in [1/2, \,1)$ are fixed.
\end{abstract}

\section{Introduction}

This paper establishes the following new results for extreme values of derivatives of the Riemann zeta
function (in this paper, we use the short-hand
notations, $\,\log_2 T :\,= \log\log T,$ and $\log_3 T :\,= \log\log\log T$ ).

\begin{theorem}\label{Main: sigma =  1} 
If $T$ is sufficiently large, then uniformly for all positive integers $\ell \leqslant (\log T) / (\log_2 T)$, we have
\[  \max_{T\leqslant t\leqslant 2T}\left|\zeta^{(\ell)}\Big(1+it\Big)\right| \geqslant e^{\gamma}\cdot \ell^{\ell}\cdot (\ell+1)^{ -(\ell+1)}\cdot\Big(\log_2 T - \log_3 T + O(1)\Big)^{\ell+1} \,. \]
\end{theorem}

\begin{remark}
In our Theorem \ref{Main: sigma =  1} , $\ell$ does not have to be fixed. In particular, if  $\ell = [(\log T)/(\log_2 T)]$, then for sufficiently large $T$, we have  \[\max_{T \leqslant t\leqslant 2T}  \left|\zeta^{(\ell)}(1 + it)\right| \gg \exp\Big\{ \frac{\log T}{\log_2 T}  (\log_3 T) -4\,\frac{\log T}{(\log_2 T)^2}   (\log_3 T)\Big\}.\]  This value is even larger than the  conditional upper bound of extreme value of the Riemann zeta function  on the $\frac{1}{2}-$line in the same interval $[T, 2T]$.  Recall that  Littlewood \cite{Littlewood} proved that the Riemann Hypothesis (RH) implies  the existence of a constant $C$ such that  for large $T$ we have $~~\max_{T \leqslant t\leqslant 2T}  \left|\zeta(\frac{1}{2} + it)\right|$ $\ll \exp\{C(\log T)$ $\cdot (\log_2 T)^{-1} \}$ .  Chandee and Soundararajan  \cite{CSo} proved that on RH, one can take any constant $C > (\log 2)/2.$

\end{remark}

\begin{theorem} \label{Main}
Let $\ell \in \mathbb{N}$ and $\beta \in [0, 1) $ be fixed.

\text{\emph{(A)}} \quad
Let $c$ be a positive number less than 
$\sqrt{2(1-\beta)}$. If $T$ is sufficiently large, then 
\begin{equation*} \label{eq:Half} \max_{T^{\beta}\leqslant t\leqslant T}\left|\zeta^{(\ell)}\Big(\frac{1}{2}+it\Big)\right| \geqslant  \exp  \Big\{c\sqrt{\frac{\log T \,\log_3 T}{\log_2 T}}\Big\}. \end{equation*}

\text{\emph{(B)}} \quad Let $\sigma \in (\frac{1}{2} , 1)$ be given and $\kappa$ be a positive number less than 
$1-\beta$. Then for  sufficiently large $T$, we have \begin{equation*} \label{eq:srtip} \max_{T^{\beta}\leqslant t\leqslant T}\left|\zeta^{(\ell)}\Big(\sigma+it\Big)\right| \geqslant   \exp  \Big\{\frac{\widetilde{c}\cdot\kappa^{1-\sigma}}{(1-\sigma)}\cdot  \frac{(\log T)^{1 - \sigma}}{(\log_2 T)^{\sigma}} \Big\}\,, \end{equation*}
  where $\widetilde{c}$ is an absolute positive constant.

\end{theorem}


The  research for extreme values of the Riemann zeta function has a long history.   In 1910,  Bohr and Landau  first established the result $\zeta(1+it) = \Omega(\log_2 t)$ (see \cite[Thm 8.5]{T}). In 1924,  Littlewood  (see \cite[Thm 8.9(A)]{T}) was able to find an explicit constant in the $\Omega$-result of  Bohr and Landau, by proving  that $\overline{\lim}_{t\to\infty} |\zeta(1+it)|/(\log_2 t) \geqslant e^{\gamma}.$  Littlewood's result was improved by Levinson \cite{L} in 1972, and by 
Granville-Soundararajan \cite{GS} in 2005. The curently best-known lower bound is established by Aistleitner-Mahatab-Munsch \cite{AMMun} in 2017,
who proved that 
$\max_{\sqrt T \leqslant t\leqslant T} \left|\zeta(1 + it)\right|  \geqslant e^{\gamma} (\log_2 T + \log_3 T - C),$
for some constant $C.$

On the other hand, when assuming Riemann hypothesis,  Littlewood proved that $|\zeta(1+it)| \leqslant (2 e^{\gamma} + o(1)) \log_2 t\,, $ for sufficiently large $t$ (see \cite[Thm 14.9]{T}). Furthermore, Littlewood conjectured that $ \max_{1 \leqslant t\leqslant T}  \left|\zeta(1 + it)\right| \sim e^{\gamma}\log_2 T .$   In \cite{GS},  Granville-Soundararajan  made the  stronger conjecture:\,
$
  \max_{T \leqslant t\leqslant 2T}  \left|\zeta(1 + it)\right| = e^{\gamma}(\log_2 T + \log_3 T + C_1) + o(1), 
$
for some constant $C_1$ which can be effectively computed.

Comparing  to the research on  extreme values of the  Riemann zeta function,    much less is known about the extreme values of its derivatives.   

It is still uncertain whether the methods of \cite{T, L, GS, AMMun} are able to establish the result in our Theorem \ref{Main: sigma =  1}, since those methods basically rely on the fact that the $k$-divisors function $d_k(n)$ is  multiplicative  and/or the fact that the Riemann  zeta function has a Euler product: $\zeta(s) = \prod_{p}(1- p^{-s})^{-1},\,$  $\Re(s)>1.$
Note that the function $f(n):\, = (\log n)^{\ell}$ is not multiplicative and the derivative $\zeta^{(\ell)} (s)$ does not have a Euler product.

We also emphasize that the key points in  Theorem \ref{Main: sigma =  1} are the range $\ell \leqslant (\log T)/(\log_2 T)$ and the constant in front of $(\log_2 T)^{\ell+1}$. In fact, one can  use the method of Bohr-Landau to prove a much weaker result, i.e., $\zeta^{(\ell)} (1+it) = \Omega((\log_2 t)^{\ell+1})$ when $\ell \in \mathbb N$ is fixed. See Section \ref{Short} for such a short proof.

We will use  Soundararajan's original resonance method  \cite{So} to prove  Theorem \ref{Main: sigma =  1}. The new ingredient for the proof is the following  Proposition \ref{maxRatio}.

\begin{proposition}\label{maxRatio}
  If $T$ is sufficiently large, then uniformly for all positive integers $\ell$, we have
\begin{equation*}
\max_{r} \Big|\sum_{mk = n\leqslant \sqrt T} \frac{r(m)\overline{r(n)}}{ k} (\log k)^{\ell}
\Big| \Big/ \Big(\sum_{n\leqslant \sqrt T }|r(n)|^2\Big) 
\geqslant \frac{e^{\gamma}}{\ell} \cdot\Big(\frac{\ell}{\ell +1}\Big)^{\ell+1} \cdot \Big(\log_2 T - \log_3 T + O(1)\Big)^{\ell+1},
\end{equation*}
where the maximum is taken over all functions $r :\, \mathbb N \to \mathbb C$ satisfying that the denominator is not equal to zero, when the parameter $T$ is given.
\end{proposition}

The following Proposition \ref{GCD_log} will not be used to prove  our theorems. But it is closely related to  Proposition \ref{maxRatio} and  can be viewed as a ``log-type'' GCD sum, so we list it here for independent interest.

\begin{proposition}\label{GCD_log}
Let $\ell \in \mathbb{N}$  and let $c_{\ell}$ be a positive number less than  $\,\,6 e^{2\gamma}\pi^{-2}\cdot \ell^{2\ell}\cdot (2\ell+1)^{ -(2\ell+1)}\,.$\,  For sufficiently large $N$, we have 
\begin{equation*}
\max_{|\M| = N} \sum_{m, n\in \M} \frac{(m,n)}{[m,n]}\log^{\ell} \Big(\frac{m}{(m,n)}\Big)\log^{\ell}\Big(\frac{n}{(m,n)}\Big) \geqslant  c_{\ell} \cdot N\cdot(\log_2 N)^{2+2\ell}\,,
\end{equation*}
where the maximum is taken over all subsets   $\M \subset \mathbb N$ with size $N$.
\end{proposition}

\begin{remark}
Actually we can also use Proposition \ref{GCD_log} and  Hilberdink's version of the resonance method \cite{Hi} to prove a similar result to the one in Theorem \ref{Main: sigma =  1}. But the constant in front of  $\,\,(\log_2 T)^{\ell+1}$ will be much worse. 
\end{remark}
Soundararajan introduced  his resonance method in \cite{So} and proved that \[ \max_{T \leqslant t \leqslant 2T} \left|\zeta\Big(\frac{1}{2}+it\Big)\right| \geqslant \exp\left((1+o(1))\sqrt{\frac{\log T }{\log\log T}}\right),\]
which improved earlier results of  Montgomery  and Balasubramanian-Ramachandra.  Montgomery \cite{M} proved it under RH and with the constant $1/20$ instead of $1+o(1) $ in Soundararajan's result.   Balasubramanian-Ramachandra \cite{BR} proved the result unconditionally but also with a smaller constant compared to Soundararajan's result.   

By constructing large GCD sums,  Aistleitner \cite{A} used a modified version of Soundararajan's resonance method to establish lower bounds for maximum of $|\zeta(\sigma+it)|$  when $\sigma \in (1/2, 1)$ is fixed. He proved that $$\max_{0 \leqslant t \leqslant T} \left|\zeta\Big(\sigma+it\Big)\right| \geqslant \exp\left(\frac{c_{\sigma}(\log T)^{1-\sigma} }{(\log\log T)^{\sigma}}\right),$$
for large $T$, and one can take $c_{\sigma} = 0.18 (2\sigma-1)^{1-\sigma}.$ 
The same result has been proved by Montgomery in \cite{M} with a smaller value for  $c_{\sigma} $.  In \cite{BSNote}, Bondarenko and  Seip improved the value $c_{\sigma} $ in Aistleitner's result.

By constructing large GCD sums, using a convolution formula for $\zeta$ in the resonance method,  Bondarenko and  Seip \cite{BS1, BS2} proved the following surprising result:\[ \max_{1 \leqslant t \leqslant T} \left|\zeta\Big(\frac{1}{2}+it\Big)\right| \geqslant \exp\left((1+o(1))\sqrt{\frac{\log T \log_3 T}{\log_2 T}}\right).\] 

After optimizing the GCD sums, de la Bret\`eche and Tenenbaum \cite{delaBT}  improved the factor from $(1 + o(1))$   to $ (\sqrt{2} + o(1)) $ in the above result.  

Following the work of Bondarenko-Seip and 
de la Bret\`eche-Tenenbaum, we use their modified versions of resonance methods to prove  Theorem \ref{Main}. The new ingredient is our convolution formula for $1 + 2^{-s} + (-1)^{\ell}\zeta^{(\ell)}(s).$   
Throughout the paper, define the function $\F_{\ell}(s)$ as follows:
\begin{equation}\label{F_l}
    \F_{\ell}(s):\, = 1 + \frac{1}{2^s} + (-1)^{\ell}\zeta^{(\ell)}(s).
\end{equation}
Throughout the paper, also define the sequence $\{ a_{\ell }(n)\}_{n =1}^{\infty}$ as  $a_{\ell }(1) = 1, a_{\ell }(2) = 1 + (\log 2)^{\ell}$,  and $a_{\ell }(n) = (\log n)^{\ell}$ for $n \geqslant 3$. Then we have the following identity and the Dirichlet series converge absolutely 
\begin{equation}\label{DirichletS}
    \F_{\ell}(s) =\sum_{n=1}
^{\infty} \frac{a_{\ell }(n)}{n^s}, \quad \quad \Re(s)  > 1.
\end{equation}

The reason why we add the part $1 + 2^{-s}  $ is that we want to make $a_{\ell}(n) \geqslant 1$ for all $n \geqslant 1.$ Since when $\sigma \in [1/2, 1)$, the factor $(\log n)^{\ell}$ has very small influence on the log-type GCD sums compared  to the case $\sigma =1$, we will simply use the fact that $a_{\ell}(n) \geqslant 1$ and then come to the situation of optimizing GCD sums.


Let $  \sigma \in (0, 1]$ be given and let $\M \subset \N$ be a finite set. The greatest common divisors (GCD) sums $S_{\sigma}(\M)$ of $\M$ are defined as follows:
\begin{align*}
   S_{\sigma}(\M):\,=\sum_{m, n \in \mathcal{M}}\frac{(m,n)^{\sigma}}{[m, n]^{\sigma}}\,\,\, , \quad
\end{align*}
where $(m, n)$ denotes the greatest common divisor of $m$ and $n$ and $[m, n]$ denotes the least common multiple of $m$ and $n$.

The case $\sigma = 1$ was  studied by G\'{a}l \cite{G}, who proved that
\begin{align}\label{Gal} (\log_2 N)^2 \ll \max_{|\M| = N}  \frac{ S_{1}(\M)}{|\M|} \ll (\log_2 N)^2.\end{align}

The asymptotically sharp constant in \eqref{Gal} is $6 e^{2\gamma}\pi^{-2}$.  This fact was  proved by Lewko and Radziwi\l\l  \,\, in \cite{Lewko}.

Bondarenko and Seip  \cite{BS, BS1} proved the following  result for  GCD sums when $\sigma = \frac{1}{2}:$ 
\begin{align*}
\emph{\emph{exp}}  \Big\{\big(1 +o(1)\big)\sqrt{\frac{\log N \,\log_3 N}{\log_2 N}}\Big\} \ll  \max_{|\M| = N}  \frac{ S_{\frac{1}{2}}(\M)}{|\M|}  \ll  \emph{\emph{exp}}  \Big\{\big(7 +o(1)\big)\sqrt{\frac{\log N \,\log_3 N}{\log_2 N}}\Big\}.
\end{align*}

Later, based on constructions of \cite{BS, BS1}, de la Bret\`eche and Tenenbaum \cite{delaBT} optimized the result of Bondarenko-Seip and obtained the following:
\begin{align}\label{GCD: 1/2}
\max_{|\M| = N}  \frac{ S_{\frac{1}{2}}(\M)}{|\M|} = \emph{\emph{exp}}  \Big\{\big(2\sqrt{2} +o(1)\big)\sqrt{\frac{\log N \,\log_3 N}{\log_2 N}}\Big\}.
\end{align}

 Aistleitner, Berkes, and Seip \cite{ABS}
 proved the following essentially optimal result for GCD sums when $\sigma \in (\frac{1}{2}, 1),$ where  $c_{\sigma}$ and $C_{\sigma}$ are positive constants only depending on $\sigma:$
\begin{align}
\emph{\emph{exp}}  \Big\{c_{\sigma}\cdot  \frac{(\log N)^{1 - \sigma}}{(\log_2 N)^{\sigma}}\Big\} \ll  \max_{|\M| = N}  \frac{ S_{\sigma}(\M)}{|\M|} \ll \emph{\emph{exp}}  \Big\{C_{\sigma}\cdot  \frac{(\log N)^{1 - \sigma}}{(\log_2 N)^{\sigma}}\Big\}.
\end{align}
 
Moreover, in \cite[page 1526]{ABS}, they also gave an example (following ideas of \cite{G}) for the lower bound when $\sigma \in (\frac{1}{2}, 1)$. Let $N=2^r$ and let $\M$ be the set of   all square-free integers composed of the first $r$ primes. Then
 \begin{align}\label{GCD: sigma}
 \max_{|\M| = N}  \frac{ S_{\sigma}(\M)}{|\M|} \gg \emph{\emph{exp}}  \Big\{\frac{\widetilde{c}}{1-\sigma}\cdot  \frac{(\log N)^{1 - \sigma}}{(\log_2 N)^{\sigma}}\Big\}
\end{align}
 for some positive constant $\widetilde{c}$. For simplicity, in our proof we will use this construction. For more  constructions, see Bondarenko-Seip \cite{BSNote}.

\section{Lemmas for the Riemann zeta function}

\begin{lemma}\label{approx}
 Let  $\sigma_0 \in (0, 1)$ be fixed. If $T$ is sufficiently large,
then uniformly for $\varepsilon >0$, $t \in [T, 2T]$, $\sigma \in[ \sigma_0 + \varepsilon, \, \infty)$ and all positive integers $\ell$, we have
\begin{equation}
   (-1)^{\ell} \zeta^{(\ell)}(\sigma+ it) = \sum_{n \leqslant T} \frac{(\log n)^{\ell}}{n^{\sigma+it}} + O\Big(  \frac{\ell !}{\varepsilon^{\ell}}\cdot T^{-\sigma+\epsilon}\Big)\,,
\end{equation}
where the implied constant in big $O(\cdot)$ only depends on $\sigma_0$ .
\end{lemma}

\begin{proof}
It follows from  Hardy-Littlewood's classical approximation formula (see \cite[Thm 4.11]{T}) for $\zeta(s)$ and Cauchy's integral formula for derivatives.
\end{proof}

\begin{lemma}\label{Estimate: ZetaDeriva}
Let $\ell \in \mathbb{N}$ and $\epsilon \in (0, 1)$ be fixed.
Then uniformly for all   $|t| \geqslant 1$ and $\sigma \in [-\epsilon, 1+\epsilon]$, 
\begin{equation}
    \zeta^{(\ell)}(\sigma + it) \ll |t|^{\frac{1 - \sigma + 3 \epsilon}{2}} ,
\end{equation}

where the implied constant depends on $\ell$ and $\epsilon$ only.

\end{lemma}



\begin{proof}
It follows from classical convex estimates for $\zeta(s)$ and Cauchy's integral formula.
\end{proof}

In the following, we will derive a ``double version" convolution formula, similar to Lemma of 5.3 of de la Bret\`eche and Tenenbaum \cite{delaBT}. The proof is same as the proof of   ``single version" convolution formulas in Lemma 1 of   Bondarenko and  Seip \cite{BS2}.

Define the Fourier transform $\widehat{K}$ of $K$  as   
\[ \widehat{K}(\xi):\,=\int_{-\infty}^{\infty} K(x) e^{-ix\xi} dx. \]

\begin{lemma}\label{lem:Convo_F}
Let  $\ell \in \mathbb{N}$ and $\sigma\in [0, 1)$ be fixed. Write $z = x +i y.$ Assume that $K(z)$ is a holomorphic function in the strip $y=\Im  z\in [\sigma-2,0]$, satisfying the growth condition 

\begin{equation} \label{Kernel}
\max_{ \sigma-2\leqslant y\leqslant 0}\big|K(z)\big| = O(\frac{1}{x^2+1}).
\end{equation}
If $t\in \mathbb{R} \setminus \{0\}$, then

\begin{align} \label{ConvFom}
\int_{-\infty}^{\infty} \F_{\ell}(\sigma+it+iy) \F_{\ell}(\sigma-it+iy) K(y) dy
= &  \sum_{m, n \geqslant 1} \frac{\widehat{K}(\log nm)}{n^{\sigma+it} \cdot m^{\sigma-it}} a_{\ell}(n)a_{\ell}(m)   \\
\nonumber    & - 2\pi ( \Delta^{+} + \Delta^{-}) \ell! \end{align}

where 
\begin{align}
 \label{ConvResd}
 \Delta^{+} = \sum_{\substack{  m + n =\ell \\ m, n \geqslant 0}} \frac{1}{m! n!} (\frac{d}{dz})^m \F_{\ell}(z+it)\Big |_{z = 1 + it} \cdot (\frac{d}{dz})^n K(i \sigma - iz)\Big |_{z = 1 + it}
\end{align}

and 
\begin{equation}\label{ConvResd2}
\Delta^{-} = \sum_{\substack{  m + n =\ell \\ m, n \geqslant 0}} \frac{1}{m! n!} (\frac{d}{dz})^m \F_{\ell}(z-it)\Big |_{z = 1 - it} \cdot (\frac{d}{dz})^n K(i \sigma - iz)\Big |_{z = 1 - it}.
\end{equation}
 
\medskip

\end{lemma}

\begin{proof}
 Define $h(z):\, = \F_{\ell}(z+it) \F_{\ell}(z-it) K(i \sigma - iz) .$ $h(z)$ is a holomorphic function with  two poles, namely at $z = 1+it$ and $z = 1 - it$. Let $Y$ be large and consider straight line integrals for $h(z)$. Set $J_1 = \int_{\sigma - i Y}^{2 -i Y}h(z)dz,\, J_2 =\int_{2 - i Y}^{2 + i Y}h(z)dz,\,  J_3 =\int_{2 + i Y}^{\sigma + i Y}h(z)dz, \, J_4 =\int_{\sigma + i Y}^{\sigma - i Y}h(z)dz .$

Note that $\F_{\ell}(s) =  \ell !/(s-1)^{\ell+1} + E(s)$, where $E(s)$ is an entire function. The residue theorem gives that 
\begin{equation}
  J_1 + J_2 + J_3 + J_4  = 2 \pi i ( \Delta^{+} + \Delta^{-}) \ell!
\end{equation}
  
  By ($\ref{DirichletS}$), ($\ref{Kernel}$)  and  applying Cauchy’s theorem term by term, \[\lim_{Y \to \infty} J_2 = i  \sum_{ n = 1}^{\infty}\sum_{m = 1}^{\infty}a_{\ell}(n)a_{\ell}(m) \frac{\widehat{K}(\log nm)}{n^{\sigma+it} \cdot m^{\sigma-it}}. \]
  
  Clearly,
\[
  \lim_{Y \to \infty} (-J_4) = i \int_{-\infty}^{\infty} \F_{\ell}(\sigma+it+iy) \F_{\ell}(\sigma-it+iy) K(y) dy.
  \]
  
By the trivial estimate $F_{\ell}(s) \ll 1 + |\zeta^{(\ell)}(s)|$, estimates for $\zeta^{(\ell)}(s)$ (Lemma \ref{Estimate: ZetaDeriva}) and (\ref{Kernel}), we obtain \[
    J_1 \ll_{\epsilon} \frac{1}{Y^2} \Big(1+\int_{\sigma}^{1+\epsilon}(1+ Y^{1-x+3\epsilon}) dx    \Big)
    \ll_{\epsilon} \frac{1}{Y^2}\Big( 1 + \frac{Y^{1+3\epsilon - \sigma} - Y^{2 \epsilon}}{\log Y} \Big).\]
 
 Take $\epsilon = 1/6$, then $J_1 \ll 1/( \sqrt Y \log Y) \to 0$, as $Y \to \infty$. Similarly for $J_3.$

\end{proof}


The following results are due to  Hadamard, Landau and Schnee (also see \cite{Ingham}).

\begin{lemma}[Hadamard, Landau, Schnee]\label{moment1}
 Let $\mu, \nu \in \mathbb{N}$ and $\alpha_1, \alpha_2 \in (- \frac{1}{2}, \infty)$ be fixed. Suppose $\alpha_1 + \alpha_2 > 1$, then  
\begin{equation*}
    \int_1^T \zeta^{(\mu)}(\alpha_1\!+\! it)\,\zeta^{(\nu)}(\alpha_2\!-\! it) dt \sim \zeta^{(\mu + \nu)}(\alpha_1 + \alpha_2) T.
\end{equation*}
\end{lemma}

In particular, when $\ell \in \mathbb{N}$ and $\sigma \in (\frac{1}{2}, 1)$ are fixed, one has
\begin{equation}
  \int_0^T \big|\zeta^{(\ell)}(\sigma\!+\!it)\big| ^2 dt \sim \zeta^{(2\ell)}(2\sigma) T. \quad 
\end{equation}






For $\sigma = \frac{1}{2}$,   Ingham \cite[ page 294, Theorem \textbf{A"} ]{Ingham} has proved the following result on second moments  of $\zeta^{(\ell)}(s)$. 
\begin{lemma}[Ingham]\label{moment2}
 Let $\ell \in \mathbb{N}$  be fixed. Then 
\begin{equation*}
    \int_0^T \big|\zeta^{(\ell)}(\frac{1}{2}\!+\! it)\big| ^2 dt \sim  \frac{T}{2\ell+1} (\log \frac{T}{2\pi})^{2\ell+1}.
\end{equation*}
\end{lemma}

\medskip

\section{Proof of  Proposition \ref{maxRatio}}

\begin{proof}
We will use the construction of Bondarenko and  Seip in \cite{BSNote}. 

Let $\delta = \ell\cdot(\ell+1)^{-1}$. Given a positive number $y$ and a positive integer $b$, define 
\[\PP(y, b):\, = \prod_{p \leqslant y} p^{b-1}\,\;.\]

We will choose a number $x$ and an integer $b$ later to make $\PP(x, b) \leqslant \sqrt T$. Let $\M$ be the set of divisors of $\PP(x, b)$ and $\M_{\delta}$ be the set of divisors of $\PP(x^{\delta}, b)$. Let $\,\,\overline{\M_{\delta}}\,\,$ be the  complement of $\M_{\delta}$ in $\M$. Note that both $\M$ and $\M_{\delta}$ are divisor-closed which means $k|n, n\in \M \implies k \in \M $ and $k|n, n\in \M_{\delta} \implies k \in \M_{\delta} $. Define the function $r:\,\N \to \{0,\,1\}$ to be the characteristic function of $\M$, then 
\begin{align*}
    \Big|\sum_{mk = n\leqslant \sqrt T} \frac{r(m)\overline{r(n)}}{ k} (\log k)^{\ell}
\Big| \Big/ \Big(\sum_{n\leqslant \sqrt T }|r(n)|^2\Big) = \frac{1}{|\M|}\sum_{mk = n \in \M}\frac{(\log k)^{\ell}}{k} = \frac{1}{|\M|}\sum_{\substack{ n \in \M\\ k|n}}\frac{(\log k)^{\ell}}{k}\,.
\end{align*}

As showed in \cite{BSNote}, 
\begin{align*}
    \frac{1}{|\M|}\sum_{\substack{ n \in \M\\ k|n}}\frac{1}{k} =  \prod_{p \leqslant x}\Big(1+ \sum_{\nu = 1}^{b-1}\Big(1-\frac{\nu}{b}\Big)p^{-\nu}\Big)\,.
\end{align*}

Also in \cite{BSNote}, it is proved that
\begin{align}\label{Product}
    \prod_{p \leqslant x}\Big(1+ \sum_{\nu = 1}^{b-1}\Big(1-\frac{\nu}{b}\Big)p^{-\nu}\Big) =\Big( 1 + O(b^{-1}) + O\big( \frac{1}{\sqrt x \log x} \big) \Big)\,e^{\gamma} \log x\,.
\end{align}

Next, we split the sum into the following two parts:
\begin{align*}
    \frac{1}{|\M|}\sum_{\substack{ n \in \M\\ k|n}}\frac{1}{k} =  \frac{1}{|\M|}\sum_{\substack{ n \in \M\\ k|n\\ k \in \M_{\delta}}}\frac{1}{k} + \frac{1}{|\M|}\sum_{\substack{ n \in \M\\ k|n\\ k \in \overline{\M_{\delta}}}}\frac{1}{k}\,.
\end{align*}

We will prove the following identity:
\begin{align}\label{M_delta}
     \frac{1}{|\M|}\sum_{\substack{ n \in \M\\ k|n\\ k \in \M_{\delta}}}\frac{1}{k} = \prod_{p \leqslant x^{\delta}}\Big(1+ \sum_{\nu = 1}^{b-1}\Big(1-\frac{\nu}{b}\Big)p^{-\nu}\Big)\,.
\end{align}

To see this, let $m$ be the largest integer such that $p_m \leqslant x^{\delta}$ and let $w$ be the largest integer such that $p_w \leqslant x $ \;($p_n$ denotes the $n$-th prime). Then we have
\begin{align*}
    \sum_{\substack{ n \in \M\\ k|n\\ k \in \M_{\delta}}}\frac{1}{k} &= \sum_{k \in \M_{\delta} }\frac{1}{k} \Big( \sum_{\substack{k|n\\ n\in \M}} 1\Big) \\&= \sum_{\alpha_1 = 0}^{b-1}\sum_{\alpha_2 = 0}^{b-1}\cdots\sum_{\alpha_m = 0}^{b-1} \frac{1}{p_1^{\alpha_1} p_2^{\alpha_2} \cdots p_m^{\alpha_m}   } (b - \alpha_1)(b - \alpha_2)\cdots (b - \alpha_m)\cdot b^{w-m} \\&= b^{w-m} \cdot \prod_{n = 1}^{ m} \Big( \sum_{\alpha_n = 0}^{ b -1} \frac{b - \alpha_n}{p_n^{\alpha_n}} \Big) \\&= b^{w} \prod_{n=1}^{m}\Big(\sum_{\nu = 0}^{b-1} \Big(1-\frac{\nu}{b}\Big)p_n^{-\nu} \Big)\,.
\end{align*}

Note that $|\M| = b^{w}$, then we immediately get  $\eqref{M_delta}$.  Now $\eqref{Product}$ together with $\eqref{M_delta}$ give that
\begin{align}
     \frac{1}{|\M|}\sum_{\substack{ n \in \M\\ k|n\\ k \in \M_{\delta}}}\frac{1}{k} = \prod_{p \leqslant x^{\delta}}\Big(1+ \sum_{\nu = 1}^{b-1}\Big(1-\frac{\nu}{b}\Big)p^{-\nu}\Big)  =\Big( 1 + O(b^{-1}) + O\big( \frac{1}{ \sqrt{x^{\delta}}  \log x} \big) \Big)\,e^{\gamma}\cdot \delta \cdot \log x\,,
\end{align}
where we omit the term $\delta^{-1}$ inside the second big $O(\cdot)$ term since $1<\delta^{-1} \leqslant 2$. Thus we obtain
\begin{align*}
 \frac{1}{|\M|}\sum_{\substack{ n \in \M\\ k|n\\ k \in \overline{\M_{\delta}}}}\frac{1}{k}   = \frac{1}{|\M|}\sum_{\substack{ n \in \M\\ k|n}}\frac{1}{k} - \frac{1}{|\M|}\sum_{\substack{ n \in \M\\ k|n\\ k \in \M_{\delta}}}\frac{1}{k} = \Big( 1 + O(b^{-1}) + O\big( \frac{1}{ \sqrt{x^{\delta}}  \log x} \big) \Big)\,e^{\gamma} (1-\delta)\, \log x\,.
\end{align*}

By the definition of $\,\overline{\M_{\delta}}\,$, if $k\in \,\,\overline{\M_{\delta}}\,$, then $\log k \geqslant \delta\,\log x$ . So we have
\begin{align*}
    \frac{1}{|\M|}\sum_{\substack{ n \in \M\\ k|n}}\frac{(\log k)^{\ell}}{k} \geqslant \frac{1}{|\M|}\sum_{\substack{ n \in \M\\ k|n\\ k \in \overline{\M_{\delta}}}}\frac{(\log k)^{\ell}}{k} \geqslant \Big( 1 + O(b^{-1}) + O\big( \frac{1}{ \sqrt{x^{\delta}}  \log x} \big) \Big)\,e^{\gamma} (1-\delta)\delta^{\ell}\, (\log x)^{\ell+1}\,.
\end{align*}
Now we set $x = (\log T) / (3 \log_2 T)$ and $b = [\log_2 T]$. By the prime number theorem, $\PP(x, b) \leqslant \sqrt T$  when $T$ is sufficiently large. Take the choices of $x, b$ and $\delta = \ell\cdot(\ell+1)^{-1}$ into the above inequality, then we are done.
\end{proof}

\section{Proof of theorem \ref{Main: sigma =  1}}
\begin{proof}
Set $N = [T^{\frac{1}{2}}]$ and let $R(t):\, = \sum_{n\leqslant N}r(n)n^{-it}$. Define the moments as follows:
\begin{align*}
 M_1(R,T):\, &= \int_{T}^{2T} |R(t)|^2 \Phi(\frac{t}{T} ) dt,\\   M_2(R,T):\, &= \int_{T}^{2T} (-1)^{\ell} \zeta^{(\ell)}(1 + it) |R(t)|^2 \Phi(\frac{t}{T} ) dt\,.
\end{align*}

As in \cite{So},  $\Phi:\, \mathbb R \to \mathbb R$ denotes a smooth function, compactly supported in $[1,2]$, 
with $0 \leqslant \Phi(y) \leqslant  1$ for all $y$, and $\Phi(y)=1$ for $5/4\leqslant y\leqslant 7/4$.  Partial integration gives that
${\hat \Phi}(y) \ll_{\nu} |y|^{-\nu}$ for any positive integer $\nu$.

Also in \cite{So},  Soundararajan  proved that
\begin{equation}\label{firstMoment}
    M_1(R,T) =T {\hat \Phi}(0) (1+O(T^{-1})) \sum_{n\leqslant N} |r(n)|^2\, .
\end{equation}

Since $\Phi$ is compactly supported in $[1,2]$, we deduce that
\begin{align*}
    \int_{T}^{2T}|R(t)|^2\sum_{k \leqslant T}\frac{(\log k)^{\ell}}{k^{1+it}}\Phi(\frac{t}{T})dt = T\sum_{m,\, n \leqslant N}\sum_{k \leqslant T} \frac{r(m)\overline{r(n)}}{k}(\log k)^{\ell}\cdot {\hat \Phi}\Big(T\cdot \log \frac{km}{n}\Big)\,.
\end{align*}

Since $N\leqslant T^{\frac{1}{2}}$,  for the off-diagonal terms $km\neq n$ we 
have ${\hat \Phi}(T\log (km/n))\ll T^{-2}$, by the rapid decay of ${\hat \Phi}$ ~(see \cite[page 471]{So}).  
 Thus the contribution of the off-diagonal terms $km \neq n$ to the above summands can be bounded by
\begin{align*}
    \ll T  \Big(\sum_{n\leqslant N}|r(n)|\Big)^2 \cdot  \sum_{k \leqslant T}\frac{(\log k)^{\ell}}{k}\cdot T^{-2} \ll T^{-1} (\log T)^{\ell+1} N \sum_{n\leqslant N}|r(n)|^2\,.
\end{align*}

Again, by $N = [T^{\frac{1}{2}}]$, we obtain
\begin{align}\label{DiriInt}
   \int_{T}^{2T}|R(t)|^2\sum_{k \leqslant T}\frac{(\log k)^{\ell}}{k^{1+it}}\Phi(\frac{t}{T})dt = &\, {\hat \Phi}(0)T\sum_{mk = n\leqslant \sqrt T} \frac{r(m)\overline{r(n)}}{ k} (\log k)^{\ell}\\
 \nonumber &+ \,O\Big( T^{-\frac{1}{2}}   (\log T)^{\ell+1}  \sum_{n\leqslant \sqrt T}|r(n)|^2\Big)\,.
\end{align}

By Lemma \ref{approx}, we have the following approximation formula and the implied constant in the big $O(\cdot)$ term is absolute:
\begin{equation*}
   (-1)^{\ell} \zeta^{(\ell)}(1+ it) = \sum_{k \leqslant T} \frac{(\log k)^{\ell}}{k^{1+it}} + O\Big(  \frac{\ell !}{\epsilon^{\ell}}\cdot T^{-1+\epsilon}\Big)\,,\quad T \leqslant t \leqslant 2T.
\end{equation*}

In the integral of $M_2(R, T)$, the big $O(\cdot)$ term above contributes  at most
\begin{align*}
    \ll \int_{T}^{2T}\frac{\ell !}{\epsilon^{\ell}}\cdot T^{-1+\epsilon} \cdot \big|R(t)\big|^2 \Phi(\frac{t}{T} ) dt \ll \frac{\ell !}{\epsilon^{\ell}}\cdot T^{-1+\epsilon} \cdot  M_1(R, T) \, .
\end{align*}

Combining this with  (\ref{DiriInt}), we have
\begin{align*}
    M_2(R, T) =  & {\hat \Phi}(0)T\sum_{mk = n\leqslant \sqrt T} \frac{r(m)\overline{r(n)}}{ k} (\log k)^{\ell}  + \,O\Big( T^{-\frac{1}{2}}   (\log T)^{\ell+1}  \sum_{n\leqslant \sqrt T}|r(n)|^2\Big)\\&+   O\Big(  \frac{\ell !}{\epsilon^{\ell}}\cdot T^{-1+\epsilon}\Big) \cdot  M_1(R, T)\,.
\end{align*}

 Finally, the above formula together with $\eqref{firstMoment}$ give that
\begin{align*}
\max_{T\leqslant t\leqslant 2T}\left|\zeta^{(\ell)}\Big(1+it\Big)\right| & \geqslant \frac{|M_2(R, T)|}{M_1(R, T)} \\ \geqslant &\big(1 + O(T^{-1})\big)
  \Big|\sum_{mk = n\leqslant \sqrt T} \frac{r(m)\overline{r(n)}}{ k} (\log k)^{\ell}
\Big| \Big/ \Big(\sum_{n\leqslant \sqrt T }|r(n)|^2\Big) \\
&+\,O\Big( T^{-\frac{3}{2}}\, (\log T)^{\ell +1} \Big)  + O\Big(  \frac{\ell !}{\epsilon^{\ell}}\cdot T^{-1+\epsilon}\Big) \,. \end{align*}
    
Now let $\epsilon = (\log_2 T)^{-1}$. By  Stirling's formula,  if $T$ is sufficiently large, then for all positive integers $\ell \leqslant (\log T) (\log_2 T)^{-1}$, we have \,$ \ell ! \cdot \epsilon^{-\ell} \cdot T^{-1+\epsilon} \leqslant 3 (\log_2 T)^{\ell}$. Other big $O(\cdot)$ terms  can be easily bounded. Together with   Proposition \ref{maxRatio}, we finish the proof of Theorem \ref{Main: sigma =  1}.

\end{proof}

\section{Proof of  theorem \ref{Main}}

\subsection{Constructing the resonator}\label{sec:res}
 
Given a set $\mathcal{M}$  of positive integers and a parameter $T$, we will construct a  resonator $R(t)$, following ideas from \cite{A}, \cite{BS1} and \cite{delaBT}.
Define \[\mathcal{M}_j:\,= \Big[(1+\frac{\log T}{T})^j,(1+\frac{\log T}{T})^{j+1}\Big)\bigcap \mathcal{M} \quad (j\geqslant 0).\]
Let $\mathcal{J}$ be the set of integers $j$ such that $  \mathcal{M}_j \neq \emptyset$
and let $m_j$ be the minimum of $\mathcal{M}_j$  for $j\in \mathcal{J}$. We then set
\[ \mathcal{M}':\, = \big \{ m_j: \ j\in \mathcal{J} \big\}\]
and \[ r(m_j):\,=\sqrt{ \sum_{m\in\mathcal{M}_j} 1} = \sqrt{|\mathcal{M}_j| } \] 
for every $m_j$ in $\mathcal{M}'$. Then the  resonator $R(t)$ is defined as follows: 
\begin{equation}
   R(t):\,=\sum_{m\in \mathcal{M}'}\frac{r(m)}{m^{it}} \,.
\end{equation}

By Cauchy's inequality, one has the following trivial estimates \cite{delaBT}:
$$R(0)^2\leqslant   N \sum_{m\in \mathcal{M}'} r(m)^2\leqslant N  |\mathcal{M}| \leqslant N^2\,.$$
  \par
 
As in \cite{BS1} , set  $\Phi(t):\,= e^{-t^2/2}$. Its Fourier transform satisfies $\widehat \Phi(\xi)=\sqrt{2\pi} \Phi(\xi).$

Replacing $T$ by $  T/\log T$ in Lemma 5 of \cite{BS2}, gives that
 \begin{equation}\label{Rt2phidt}
   \int_{-\infty}^{\infty} |R(t)|^2\Phi\Big(\frac{t\log T}{T}\Big)dt \ll \frac{T|\mathcal{M}|}{\log T}  \,.
 \end{equation}

\subsection{The proof}

\begin{proof}

 Let $\sigma \in [\frac{1}{2}, 1)$. Choose $\kappa \in (0,  1-\beta)$ and  set $N:\,=[T^{\kappa}].$ Fix $\varepsilon>0$ such that $ \kappa  + 4\varepsilon <1 $.

As in \cite{BS2}, choose
\[ K(t):\,=\frac{\sin^2((\varepsilon \log T)t)}{(\varepsilon \log T)t^2},\]
which has Fourier transform 
\begin{equation}\label{Sin:fourier} \widehat{K}(\xi)=\pi \max\left(\left(1-\frac{|\xi |}{2 \varepsilon \log T}\right), 0\right). \end{equation} 

Define 
\begin{align*}
   \mathscr{Z}_{\sigma}(t, y):\,&= \F_{\ell}(\sigma+it+iy) \F_{\ell}(\sigma-it+iy) K(y)\,,\\
   I(T)  :\;&= \int_{|t| \geqslant \A}|R(t)|^2\Phi\Big(\frac{t\log T }{T}  \Big)\int_{-\infty}^{\infty}\mathscr{Z}_{\sigma}(t, y) dydt\,.
\end{align*}

Following \cite{BS2} and \cite{delaBT}, we will show that the integral on $2 T^{\beta} \leqslant |t|\leqslant \frac{T}{2} $ and $|y|\leqslant\frac{|t|}{2}$  gives the main term for $I(T)$. We will frequently use the following  trivial estimates ( Lemma \ref{Estimate: ZetaDeriva} ) :
\begin{equation}\label{Zeta3/5}
    |\F_{\ell}(\sigma \pm it+iy)|\ll 1+ |\zeta^{(\ell)}(\sigma \pm it+iy)|\ll (1 + |t|+|y|)^{\frac{3}{10}}.
\end{equation}

A simple computation gives
\[ \int_{\A \leqslant |t|\leqslant 2 T^{\beta}} \int_{|y|>  T^{\beta}}\mathscr{Z}_{\sigma}(t, y) dydt \ll \int_{|t|\leqslant 2 T^{\beta}} \int_{|y|>  T^{\beta}}(1 + |t|+|y|)^{\frac{3}{5}}\frac{1}{(|t|+|y|)^2}dydt \ll (T^{\beta})^{\frac{3}{5}}\,.
\]

Note that
\begin{align*}
&\F_{\ell}(\sigma+it+iy) \F_{\ell}(\sigma-it+iy)\\\ll &\Big(1+ |\zeta^{(\ell)}(\sigma + it+iy)|\Big)\Big(1+ |\zeta^{(\ell)}(\sigma - it+iy)|\Big)\\
\ll&1+ |\zeta^{(\ell)}(\sigma + it+iy)|+ |\zeta^{(\ell)}(\sigma - it+iy)| + |\zeta^{(\ell)}(\sigma + it+iy)|^2+ |\zeta^{(\ell)}(\sigma - it+iy)|^2  \,.
\end{align*}

Thus

\begin{align*}
\int_{\A \leqslant |t|\leqslant 2 T^{\beta}} \int_{-\infty}^{\infty}\mathscr{Z}_{\sigma}(t, y) dydt &\ll T^{\beta} + \int_{\A \leqslant|t|\leqslant 2 T^{\beta}} \int_{|y|\leqslant  T^{\beta}}\mathscr{Z}_{\sigma}(t, y) dydt\\
&\ll    T^{\beta} +  \int_{-3T^{\beta}}^{3T^\beta} |\zeta^{(\ell)}(\sigma + i t)| dt  + \int_{-3T^{\beta}}^{3T^\beta} |\zeta^{(\ell)}(\sigma + i t)|^2 dt \\
&\ll T^\beta \cdot (\log T)^{2\ell +1},
\end{align*} 

where the last step follows from Lemma \ref{moment1}, \ref{moment2} and the Cauchy–Schwarz inequality. Trivially, by $|R(t)| \leqslant R(0)$ and $\Phi(\cdot)\leqslant 1$,
\[\int_{\A \leqslant |t|\leqslant 2T^{\beta}}|R(t)|^2\Phi\Big(\frac{t\log T }{T}  \Big)\int_{-\infty}^{\infty}\mathscr{Z}_{\sigma}(t, y) dydt \ll R(0)^2 T^\beta \cdot (\log T)^{2\ell +1}  \ll |\mathscr{M}| T^{\beta + \kappa} (\log T)^{2\ell +1}.\]

 The fast decay of $\Phi$ and $(\ref{Zeta3/5})$  give that
\[\int_{|t|> \frac{T}{2}}|R(t)|^2\Phi\Big(\frac{t\log T }{T}  \Big)\int_{-\infty}^{\infty}\mathscr{Z}_{\sigma}(t, y) dydt \ll T^{\kappa +4} e^{-\frac{1}{16}(\log T)^2} \cdot  |\mathscr{M}| \ll o(1) \cdot  |\mathscr{M}| \,.
\]

Using $(\ref{Rt2phidt})$ and $(\ref{Zeta3/5})$, one can compute
\[\int_{2 T^{\beta} \leqslant |t|\leqslant \frac{T}{2}}|R(t)|^2\Phi\Big(\frac{t\log T }{T}  \Big)\int_{|y|>\frac{|t|}{2}}\mathscr{Z}_{\sigma}(t, y) dydt \ll \frac{T^{1-\frac{2}{5}\beta}}{(\log T)^2} \cdot |\mathscr{M}| \,.
\]

Combining the above estimates, one gets
\begin{align*}
I(T) = \int_{2 T^{\beta} \leqslant |t|\leqslant \frac{T}{2}}|R(t)|^2\Phi\Big(\frac{t\log T }{T}  \Big)\int_{|y|\leqslant\frac{|t|}{2}}\mathscr{Z}_{\sigma}(t, y) dydt &+ |\mathscr{M}|\cdot O\Big( T^{\beta + \kappa}  (\log T)^{2\ell +1} \Big)\\
&+  |\mathscr{M}| \cdot O\Big(\frac{T^{1-\frac{2}{5}\beta}}{(\log T)^2}   \Big)    \,.
\end{align*}

Note that $2 T^{\beta} \leqslant |t|\leqslant \frac{T}{2} $ and $|y|\leqslant\frac{|t|}{2}$  give $T^{\beta} \leqslant |t\pm y| \leqslant T.$\, Again, by $(\ref{Rt2phidt})$
\begin{align}\label{UpperBoundIT}
    I(T) \ll \frac{T|\mathcal{M}|}{\log T}  \cdot \max_{T^{\beta}\leqslant t\leqslant   T}\left|  \F_{\ell}\Big(\sigma+it\Big)  \right|^2 + |\mathscr{M}|\cdot O\Big( T^{\beta + \kappa}  (\log T)^{2\ell +1} \Big) +  |\mathscr{M}| \cdot O\Big(\frac{T^{1-\frac{2}{5}\beta}}{(\log T)^2}   \Big) .
\end{align}

Next, let 

\begin{equation}
    G_{\sigma}(t):\,= \sum_{m, n \geqslant 1} \frac{\widehat{K}(\log nm)}{n^{\sigma+it} \cdot m^{\sigma-it}} a_{\ell}(n)a_{\ell}(m)  \end{equation}

and set
\begin{align*}
I_1(T) :\;&= \int_{|t| \geqslant \A}  G_{\sigma}(t)|R(t)|^2\Phi\Big(\frac{t\log T }{T}  \Big)dt\\
 I_2(T)  :\;&= - 2\pi  \ell! \int_{|t| \geqslant \A}  \Delta^{+}\cdot|R(t)|^2\Phi\Big(\frac{t\log T }{T}  \Big)dt\\
 I_3(T) :\;&= - 2\pi  \ell! \int_{|t| \geqslant \A}  \Delta^{-}\cdot|R(t)|^2\Phi\Big(\frac{t\log T }{T}  \Big)dt \,.
\end{align*}

By the convolution formula $(\ref{ConvFom})$, one obtains
\[I(T) = I_1(T) +I_2(T)+ I_3(T).\]

We will bound  $I_2(T), I_3(T)$ as follows:
\begin{align}\label{I2I3}
    |I_2(T)| + |I_3(T)| \ll |\M|\cdot \frac{T^{\kappa+\frac{5}{4}\varepsilon}}{\log T}\,.
\end{align}

By Cauchy's integral for derivatives and the explicit expression for $K$,  we have the following estimates for all $0 \leqslant n  \leqslant l,$
\begin{align}
    \Big(\frac{d}{dz}\Big)^n K(i \sigma - iz)\Big |_{z = 1 - it} \ll \max_{|\alpha| = \frac{1}{8}}\Big | K\Big(i \sigma -i(1 - it)+\alpha\Big)\Big | \ll \frac{T^{\frac{5}{4}\varepsilon}}{\log T \cdot | t |^2},\quad \forall |t|\geqslant \A,
\end{align}
where the implied constants depend on $\varepsilon$ and $\ell$ only.

And trivially, for all $0 \leqslant m  \leqslant l$, one has
\begin{align}
   \Big(\frac{d}{dz}\Big)^m \F_{\ell}(z-it)\Big |_{z = 1 - it} \ll 1+ \Big|\zeta^{(\ell+ m)}(1- 2it)\Big| \ll |t|^{\frac{1}{8}}, \quad \forall |t|\geqslant \A,
\end{align}
where the implied constants depend only on $\ell$.

Note that there are  finitely many non-negative integer  pairs  $(m,\,n)$ satisfying $m +n = \ell$, so
\begin{align}
    I_3(T) \ll R(0)^2 \int_{|t| \geqslant \A}  \frac{T^{\frac{5}{4}\varepsilon} \cdot |t|^{\frac{1}{8}}}{\log T \cdot | t |^2} dt \ll  (T^{\kappa} \cdot |\mathscr{M}|) \cdot  \frac{T^{\frac{5}{4}\varepsilon} }{\log T }\int_{|t| \geqslant \A} \frac{ |t|^{\frac{1}{8}}}{| t |^2} dt \ll |\M|\cdot \frac{T^{\kappa+\frac{5}{4}\varepsilon}}{\log T}\,.
\end{align}
Proceed similarly for  $I_2(T)$, so we get \eqref{I2I3}.

 Next, in order to relate $I_1(T)$ to the GCD sums,  we would like to use Fourier transform on the whole real line. So set
\begin{align*}
    \widetilde{I_1}(T) :\;&= \int_{-\infty}^{\infty}  G_{\sigma}(t)|R(t)|^2\Phi\Big(\frac{t\log T }{T}  \Big)dt\,.
\end{align*}
By $(\ref{Sin:fourier})$, $ \widehat{K}(\log nm) = 0 $ if $ mn \geqslant T^{2 \varepsilon}.$ Clearly, $a_l(n)/n^{\sigma}\ll 1.$ So one can get
\begin{align*}
    \int_{|t| \leqslant \A}  G_{\sigma}(t)|R(t)|^2\Phi\Big(\frac{t\log T }{T}  \Big)dt &= \int_{|t| \leqslant \A}  \Big(\sum_{m, n \geqslant 1} \frac{\widehat{K}(\log nm)}{n^{\sigma+it} \cdot m^{\sigma-it}} a_{\ell}(n)a_{\ell}(m)\Big)\Big|R(t)\Big|^2\Phi\Big(\frac{t\log T }{T}  \Big)dt\\
    &\ll R(0)^2  \sum_{m, n \geqslant 1} \frac{\widehat{K}(\log nm)}{n^{\sigma} \cdot m^{\sigma}} a_{\ell}(n)a_{\ell}(m)\\
    &\ll (T^{\kappa} \cdot |\mathscr{M}|)\cdot \sum_{mn \leqslant T^{2\varepsilon}} 1\\
     &\ll T^{\kappa+ 4\varepsilon} \cdot |\mathscr{M}|\,.
\end{align*}
We obtain $I_1(T)=\widetilde{I_1}(T)+ |\mathscr{M}| \cdot O(T^{\kappa+ 4\varepsilon})$.

Thus

\begin{align}\label{I1: Upper}
    \widetilde{I_1}(T) \ll \frac{T|\mathcal{M}|}{\log T}  \cdot \max_{T^{\beta}\leqslant t\leqslant   T}\left|  \F_{\ell}\Big(\sigma+it\Big)  \right|^2 +  |\mathscr{M}| \cdot O(T^{\kappa+ 4\varepsilon}) &+ |\mathscr{M}|\cdot O\Big( T^{\beta + \kappa}  (\log T)^{2\ell +1} \Big)\\ \nonumber &+  |\mathscr{M}| \cdot O\Big(\frac{T^{1-\frac{2}{5}\beta}}{(\log T)^2}   \Big) .
\end{align}

We compute the integral $\widetilde{I_1}(T)$ by expanding the product of the  resonator and the infinite series of $G_{\sigma}(t)$, then integrate  term by term, as in \cite[page 1699]{BS1}. Using the fact $a_{\ell}(k) \geqslant 1$ for every $k$ and $\widehat{K}(\log jk) \geqslant \pi/2$ if $ jk \leqslant T^{\varepsilon}$, one gets
\begin{align*}  \widetilde{I_1}(T) & = \frac{T\sqrt{2\pi}}{\log T}\sum_{m, n \in \mathcal{M}' }r(m)r(n)\sum_{j,k \geqslant 1}a_{\ell}(j)a_{\ell}(k)\frac{\widehat{K}(\log jk)}{(jk)^{\sigma}}\Phi\Big(\frac{T}{\log T} \log \frac{mj}{nk}\Big)\\
& \geqslant \frac{T\sqrt{2\pi}}{\log T}\sum_{m, n \in \mathcal{M}' }r(m)r(n)\sum_{j,k \geqslant 1}\frac{\widehat{K}(\log jk)}{(jk)^{\sigma}}\Phi\Big(\frac{T}{\log T} \log \frac{mj}{nk}\Big)\\
& \gg \frac{T}{\log T}\sum_{1 \leqslant jk \leqslant T^{\varepsilon}}\frac{1}{(jk)^{\sigma}}\sum_{m, n \in \mathcal{M}' }r(m)r(n)\Phi\Big(\frac{T}{\log T} \log \frac{mj}{nk}\Big)\,.
\end{align*}

Next, proceed as in \cite{delaBT} (following ideas from \cite{BS1}),
\begin{align}\label{I1: Lower}
    \widetilde{I_1}(T) \gg \frac{T}{\log T} \sum_{\substack{m, n \in \mathcal{M}\\ \frac{[m, n]}{(m, n)}\leqslant T^{\varepsilon}}}\frac{(m,n)^{\sigma}}{[m, n]^{\sigma}} \gg \frac{T}{\log T}\Big(S_{\sigma}(\M)- T^{\varepsilon(\frac{1}{3} - \sigma)} \cdot S_{\frac{1}{3}}(\M) \Big)\,.
\end{align}

Combining $(\ref{I1: Upper})$ with $(\ref{I1: Lower})$, we have 
\begin{align}\label{Max: Fl}
    &\max_{T^{\beta}\leqslant t\leqslant   T}\left|  \F_{\ell}\Big(\sigma+it\Big)  \right|^2 +  O\Big( T^{\beta + \kappa  -1}  (\log T)^{2\ell +2} \Big) +   O\Big(T^{\kappa+ 4\varepsilon-1}\log T\Big) +    O\Big(\frac{T^{-\frac{2}{5}\beta}}{\log T}   \Big)   \\ \nonumber & \gg \frac{S_{\sigma}(\M)}{|\M|} - T^{\varepsilon(\frac{1}{3} - \sigma)} \cdot \frac{  S_{1/3}(\M)}{|\M|} \,.
\end{align}

Next, we will consider the two cases $ \sigma = \frac{1}{2}$ and $ \sigma \in (\frac{1}{2}, 1) $ separately. 

\textbf{Case 1:\,  $\sigma = \frac{1}{2}$.} 

In this case, let $\M$ be the set in $(\ref{GCD: 1/2})$ with $|\M| = N $. Recall that $N =[T^{\kappa}],$ so
\begin{align}
\frac{ S_{1/2}(\M)}{|\M|} \gg \emph{\emph{exp}}  \Big\{(2\sqrt{2\, \kappa} +o(1))\sqrt{\frac{\log T \,\log_3 T}{\log_2 T}}\Big\}\,.
\end{align}

Also, in \cite[page 128]{delaBT}, de la Bret\`eche and Tenenbaum  showed that for this set $\M$, 
\begin{align}
\frac{ S_{1/3}(\M)}{|\M|} \ll \emph{\emph{exp}}\Big\{ y_{\M}^{\frac{2}{3}} \Big\}, \quad \emph{\emph{where}} \quad y_{\M} \ll (\log T)^{\frac{6}{5}}.
\end{align}

So the second term on the right-hand side of  $(\ref{Max: Fl})$ is $o(1)$. And clearly, the big $O(\cdot)$ terms in $(\ref{Max: Fl})$  can be ignored. Thus
\begin{align}
 \max_{T^{\beta}\leqslant t\leqslant T}\left|\zeta^{(\ell)}\Big(\frac{1}{2}+it\Big)\right| \gg   \max_{T^{\beta}\leqslant t\leqslant   T}\left|  \F_{\ell}\Big(\frac{1}{2}+it\Big)  \right| + O(1) \gg  \emph{\emph{exp}}  \Big\{(\sqrt{2\, \kappa} +o(1))\sqrt{\frac{\log T \,\log_3 T}{\log_2 T}}\Big\}\,.
\end{align}

\textbf{Case 2:\,  $\sigma \in (\frac{1}{2}, 1)$.} 

In this case, let $\M$ be the set in $(\ref{GCD: sigma})$ with $|\M| = N $. Again, $N =[T^{\kappa}],$ so
\begin{align}
\frac{ S_{\sigma}(\M)}{|\M|} \gg  \emph{\emph{exp}}  \Big\{\frac{\widetilde{c}}{1-\sigma}\cdot  \frac{(\log N)^{1 - \sigma}}{(\log_2 N)^{\sigma}}\Big\}  \gg  \emph{\emph{exp}}  \Big\{\frac{\widetilde{c}\cdot\kappa^{1-\sigma}}{1-\sigma}\cdot  \frac{(\log T)^{1 - \sigma}}{(\log_2 T)^{\sigma}} \Big\}.
\end{align}

For the second term on the right-hand side of $(\ref{Max: Fl})$ , we have
\begin{align*}
    T^{\varepsilon(\frac{1}{3} - \sigma)} \cdot \frac{  S_{1/3}(\M)}{|\M|} \ll T^{\varepsilon(\frac{1}{3} - \sigma)} \emph{\emph{exp}}  \Big\{C_{\frac{1}{3}}\cdot  \frac{(\kappa \log T)^{\frac{2}{3}}}{(\log_2 T)^{\frac{1}{3}}} \Big\} \ll \emph{\emph{exp}}  \Big\{ -\frac{\varepsilon}{6} \log T +   C_{\frac{1}{3}}\cdot  \frac{(\kappa \log T)^{\frac{2}{3}}}{(\log_2 T)^{\frac{1}{3}}}  \Big\} = o(1).
\end{align*}

Hence
\begin{align}
 \max_{T^{\beta}\leqslant t\leqslant T}\left|\zeta^{(\ell)}\Big(\sigma+it\Big)\right| \gg   \max_{T^{\beta}\leqslant t\leqslant   T}\left|  \F_{\ell}\Big(\sigma+it\Big)  \right| + O(1) \gg  \emph{\emph{exp}}  \Big\{\frac{\widetilde{c}\cdot\kappa^{1-\sigma}}{2(1-\sigma)}\cdot  \frac{(\log T)^{1 - \sigma}}{(\log_2 T)^{\sigma}} \Big\}.  
\end{align}
Make $\kappa$ slightly larger in the beginning then one can get $(B)$.

\end{proof}

\section{Proof of Proposition \ref{GCD_log}}

The idea of the proof is basically the same as in the proof of  Proposition \ref{maxRatio}. The new ingredient is G\'{a}l's identity. 

In this section, in order to avoid confusion about notations, we use the notation $(m\otimes n)$ for the ordered pair of $m$ and $n$.

\begin{proof}

Let $\PP(\rr, b) = p_1^{b - 1} \cdot \ldots \cdot p_{\rr}^{b - 1}$\,,
where $p_n$ denotes the $n$-th prime. Define $\M$ to be the set of divisors of $\PP(\rr, b)$, then $|\M| = b^{\rr}.$
 By G\'{a}l's identity  \cite{G},
$$
 \sum_{m, n\in \M} \frac{(m,n)}{[m,n]}  = \prod_{i \leqslant \rr}
\Big ( b + 2 \sum_{\nu = 1}^{b - 1} \frac{b - \nu}{p_i^{\nu}}
\Big )\,.
$$

 Let $\rr = [ \log \kk / \log\log \kk ]$, then $p_{\rr} \sim \log \kk$ by the prime
number theorem.  Let $b$ be the integer satisfying that
$$
b^{\rr} \leqslant \kk < (b + 1)^{\rr},
$$
then  $b^{\rr} \sim \kk$, as $\kk \rightarrow \infty$.
 Choose a set $\M' \subset \N$  such that $\M \subset \M'$ and $|\M'| = \kk.$  
 
 Following   Lewko-Radziwi\l\l  \,\,in \cite{Lewko}, we use G\'{a}l's identity for the GCD sum then split the product into two parts:
\begin{align}\label{GCD: Gal}
\nonumber \sum_{m, n\in \M} \frac{(m,n)}{[m,n]} & = b^r \prod_{i \leqslant r} \Big ( 1 + 2 \sum_{v = 1}^{b - 1} \frac{1}{p_i^v}
\cdot \Big ( 1 - \frac{v}{b} \Big ) \Big ) \\ & \geqslant (1 + o(1)) \kk
\prod_{i \leqslant \rr} \Big (1 - \frac{1}{p_i} \Big )^{-2}
\times
\prod_{i \leqslant \rr} \big (1 + 2 \sum_{v = 1}^{b - 1}
\frac{1}{p_i^v} \cdot\Big ( 1 - \frac{v}{b} \Big )
\Big ) \Big (1 - \frac{1}{p_i} \Big )^2\,.
\end{align}
By Mertens' theorem, the first product is asymptotically
equal to $(e^{\gamma} \log p_{\rr})^2 \sim (e^\gamma \log \log \kk)^2$
as $\kk \rightarrow \infty$.  The  second product converges 
as $\kk \rightarrow \infty$ to
$$
\prod_{p} \Big (1 + 2 \sum_{v = 1}^{\infty} \frac{1}{p^v}
\Big ) \Big ( 1 - \frac{1}{p} \Big )^2 = \frac{6}{\pi^2}.
$$

Next, let $~~\delta = \ell \cdot (2\ell+1)^{-1}$ and define the sets $\M_{\delta}^{(1)},\,~\M_{\delta}^{(2)}$ as follows:

$\M_{\delta}^{(1)}:\, = \Big\{(m\otimes n) \in \M \times \M \Big|\,\forall i > \rr^{\delta},\,\,  \alpha_i = \min\{\alpha_i, \beta_i\},$ ~~where $m$ and $n$ have  prime factorizations as ~~~ $ m = p_1 ^{\alpha_1} p_2 ^{\alpha_2} \cdots p_{\rr} ^{\alpha_{\rr}},\, ~~\emph{\emph{and}}~~~  n = p_1 ^{\beta_1} p_2 ^{\beta_2} \cdots p_{\rr} ^{\beta_{\rr}}\Big\}\,.$

$\M_{\delta}^{(2)}:\, = \Big\{(m\otimes n) \in \M \times \M \Big|\,\forall i > \rr^{\delta},\,\,  \beta_i = \min\{\alpha_i, \beta_i\},$ ~~where $m$ and $n$ have  prime factorizations as ~~~ $ m = p_1 ^{\alpha_1} p_2 ^{\alpha_2} \cdots p_{\rr} ^{\alpha_{\rr}},\, ~~\emph{\emph{and}}~~~  n = p_1 ^{\beta_1} p_2 ^{\beta_2} \cdots p_{\rr} ^{\beta_{\rr}}\Big\}\,.$

Then define $\M_{\delta}$ to be the union of the above two sets and $\,\,\overline{\M_{\delta}}\,\,$ to be the  complement of $\M_{\delta}$ in $\M\times \M$:
\[\M_{\delta}:\, = \M_{\delta}^{(1)} \bigcup \M_{\delta}^{(2)}\,, \quad \quad \quad \overline{\M_{\delta}}:\, = (\M \times \M)\setminus \M_{\delta} \,.\]

Now we split the GCD sum into two parts:
\begin{align}\label{split}
 \sum_{m, n\in \M} \frac{(m,n)}{[m,n]} = \sum_{(m\otimes n)\in \M_{\delta}} \frac{(m,n)}{[m,n]} + \sum_{(m\otimes n)\in \overline{\M_{\delta}}} \frac{(m,n)}{[m,n]}\,.
\end{align}

By symmetry, we have
\begin{align}\label{sym}
 \sum_{(m\otimes n)\in \M_{\delta}} \frac{(m,n)}{[m,n]}  \leqslant 2 \sum_{(m\otimes n)\in \M_{\delta}^{(1)}} \frac{(m,n)}{[m,n]}\,.
\end{align}

By the definition of $\M_{\delta}^{(1)}$ and  G\'{a}l's identity, we have
\begin{align*}
   \sum_{(m\otimes n)\in \M_{\delta}^{(1)}} \frac{(m,n)}{[m,n]} = &\sum_{\substack{i \leqslant \rr^{\delta}\\0 \leqslant \alpha_i \leqslant b -1\\0 \leqslant \beta_i \leqslant b -1}} \,\,\,\,\,\, \sum_{\substack{\rr^{\delta}< i \leqslant \rr\\0 \leqslant \alpha_i \leqslant \beta_i \leqslant b -1}}  \prod_{i \leqslant \rr^{\delta}}  p_i^{-|\alpha_i - \beta_i|} \prod_{\rr^{\delta}< i \leqslant \rr}  \frac{1}{p_i^{ \beta_i - \alpha_i }}\\
   =&     \prod_{i\leqslant \rr^{\delta}}\Big(b + 2\sum_{\nu = 1}^{b-1}\frac{b-\nu}{p_i^{\nu}} \Big)  ~~~ \cdot  \prod_{\rr^{\delta}< i \leqslant \rr}\Big(\sum_{x_i = 0}^{b-1}\frac{b-x_i}{p_i^{x_i}}\Big)\\
   = &  \prod_{i\leqslant \rr^{\delta}}\Big(b + 2\sum_{\nu = 1}^{b-1}\frac{b-\nu}{p_i^{\nu}} \Big)  ~~~ \cdot  \prod_{\rr^{\delta}< i \leqslant \rr}~\Big( \Big(\frac{1}{p_i}\Big)^{b+1} -\frac{b+1}{p_i} +b\Big) \cdot \Big( 1- \frac{1}{p_i}\Big)^{-2}\\
   = & b^{\rr} \prod_{i\leqslant \rr^{\delta}}\Big(1 + 2\sum_{\nu = 1}^{b-1}  \Big(1 - \frac{\nu}{b}\Big)p_i^{-\nu} \Big)  ~~~ \cdot  \prod_{\rr^{\delta}< i \leqslant \rr}~\Big(\frac{1}{b}\cdot \Big(\frac{1}{p_i}\Big)^{b+1} -\frac{1+\frac{1}{b}}{p_i} +1\Big) \cdot \Big( 1- \frac{1}{p_i}\Big)^{-2}\\= & b^{\rr} \prod_{i\leqslant \rr^{\delta}}\Big(1 + 2\sum_{\nu = 1}^{b-1}  \Big(1 - \frac{\nu}{b}\Big)p_i^{-\nu} \Big)\Big( 1- \frac{1}{p_i}\Big)^{2}\\  ~~ & \times   \prod_{\rr^{\delta}< i \leqslant \rr}~\Big(\frac{1}{b}\cdot  \Big(\frac{1}{p_i}\Big)^{b+1} -\frac{1+\frac{1}{b}}{p_i} +1\Big)  \times \prod_{i \leqslant \rr}~ \Big( 1- \frac{1}{p_i}\Big)^{-2}\,.
\end{align*}

Again, we have $b^{\rr} \sim \kk$, the first product converges to $6/ \pi^2$, and the third product is    asymptotically
equal to $(e^{\gamma} \log p_\rr)^2 \sim (e^\gamma \log \log \kk)^2$
as $ \kk \rightarrow \infty$.

For the second product, it can be bounded as 
\begin{align*}
     \prod_{\rr^{\delta}< i \leqslant \rr}~\Big(\frac{1}{b}\cdot  \Big(\frac{1}{p_i}\Big)^{b+1} -\frac{1+\frac{1}{b}}{p_i} +1\Big) \leqslant \prod_{\rr^{\delta}< i \leqslant \rr}~ \Big( 1- \frac{1}{p_i}\Big) \,.
\end{align*}
And by  Mertens' theorem and the prime number theorem, we have
\begin{align*}
     \lim_{\rr \to \infty} \prod_{\rr^{\delta}< i \leqslant \rr}~ \Big( 1- \frac{1}{p_i}\Big) = \delta\,.
\end{align*}

As a result, we obtain that 
\begin{align*}
\sum_{(m\otimes n)\in \M_{\delta}^{(1)}} \frac{(m,n)}{[m,n]}\leqslant \Big( \delta + o(1)\Big) \kk \cdot \frac{6}{\pi^2} \cdot  (e^\gamma \log \log \kk)^2\,.
\end{align*}
Hence by  \eqref{GCD: Gal}, \eqref{split}, \eqref{sym}, we get
\begin{align*}
    \sum_{(m\otimes n)\in \overline{\M_{\delta}}} \frac{(m,n)}{[m,n]} \geqslant \Big(1-2\delta + o(1)\Big) \kk \cdot \frac{6}{\pi^2} \cdot  (e^\gamma \log \log \kk)^2\,.
\end{align*}

By the construction of $~~~\overline{\M_{\delta}}$, ~~~if $ (m\otimes n)\in \overline{\M_{\delta}} $, then 
\begin{align*}
\log \Big(\frac{m}{(m,n)}\Big) \geqslant \Big( \delta + o(1)\Big)\cdot \log\log\kk\,, \quad \quad  \log\Big(\frac{n}{(m,n)}\Big)    \geqslant \Big( \delta + o(1)\Big)\cdot \log\log\kk\,.
\end{align*}

Thus 
\begin{align*}
&\sum_{m, n\in \M'} \frac{(m,n)}{[m,n]}\log^{\ell} \Big(\frac{m}{(m,n)}\Big)\log^{\ell}\Big(\frac{n}{(m,n)}\Big) \\
\geqslant &\sum_{m, n\in \M} \frac{(m,n)}{[m,n]}\log^{\ell} \Big(\frac{m}{(m,n)}\Big)\log^{\ell}\Big(\frac{n}{(m,n)}\Big) \\
    \geqslant &\sum_{(m\otimes n)\in \overline{\M_{\delta}}} \frac{(m,n)}{[m,n]} \log^{\ell} \Big(\frac{m}{(m,n)}\Big)\log^{\ell}\Big(\frac{n}{(m,n)}\Big) \geqslant \Big( (1-2\delta)\,\delta^{2\ell}  + o(1)\Big) \kk \cdot \frac{6}{\pi^2} \cdot  e^{2\gamma} \cdot (\log \log \kk)^{2 + 2\ell}\,. 
\end{align*}

By our choice of $~~\delta = \ell \cdot (2\ell+1)^{-1}$, we are done.
\end{proof}

\section{ A Short Proof For A Weaker Result}\label{Short}

One can  use the method of Bohr-Landau (see \cite[Thm 8.5]{T}) to prove the weaker result that $\zeta^{(\ell)} (1+it) = \Omega((\log_2 t)^{\ell+1}),$ when $\ell \in \mathbb N$ is fixed.

\begin{proof}
Write $ s= \sigma + it $. When $\sigma > 1$,
\[ (-1)^{\ell} \zeta^{(\ell)}(s) = \sum_{n = 2}^{\infty} \frac{(\log n)^{\ell}}{n^{\sigma + it}} = \sum_{n = 2}^{N} \frac{(\log n)^{\ell}}{n^{\sigma + it}} + \sum_{n = N + 1}^{\infty} \frac{(\log n)^{\ell}}{n^{\sigma + it}}. \]
For given positive integers $N$ and $q$, by Dirichlet's theorem, there exists $t \in [1,\, q^N]$, such that  $\cos{(t \log n)} \geqslant \cos{(2\pi/q)} $ for all  integers $n \in [1, N]$. Hence
\begin{align*}
 |\zeta^{(\ell)}(s)| &\geqslant \sum_{n = 2}^{N}  \frac{(\log n)^{\ell}}{n^{\sigma}} \cos{(t \log n)}\,\,  -\,\,\sum_{n = N + 1}^{\infty} \frac{(\log n)^{\ell}}{n^{\sigma }}\\  
 &\geqslant \cos{(\frac{ 2\pi}{q} )} \cdot \sum_{n = 2}^{N}  \frac{(\log n)^{\ell}}{n^{\sigma}} \,\,  -\,\,\sum_{n = N + 1}^{\infty} \frac{(\log n)^{\ell}}{n^{\sigma }}\\  
 &\geqslant \cos{(\frac{ 2\pi}{q} )} \cdot \sum_{n = 2}^{\infty}  \frac{(\log n)^{\ell}}{n^{\sigma}} \,\,  -\,\,2\sum_{n = N + 1}^{\infty} \frac{(\log n)^{\ell}}{n^{\sigma }}.
\end{align*} 

Take $q = 8$ to get 
\begin{align}\label{lower: sigma = 1}
   |\zeta^{(\ell)}(s)| &\geqslant  \cos{(\frac{ 2\pi}{8} )} \cdot \sum_{n = 2}^{\infty}  \frac{(\log n)^{\ell}}{n^{\sigma}} \,\,  -\,\,  2\sum_{n = N + 1}^{\infty} \frac{(\log n)^{\ell}}{n^{\sigma }}, \quad N\, \log 8 > \log t \,.
\end{align}

One can compute that
\begin{align}
    \sum_{n = 2}^{\infty}  \frac{(\log n)^{\ell}}{n^{\sigma}} > O_{\ell}(1) + \int_{1}^{\infty}\frac{(\log x)^{\ell}}{x^{\sigma}}dx > O_{\ell}(1) + \frac{\ell !}{(\sigma -1)^{\ell+1}}\,,
\end{align}

and for  large $N$ that
\begin{align}
   \sum_{n = N + 1}^{\infty} \frac{(\log n)^{\ell}}{n^{\sigma }} <  \int_{N}^{\infty}\frac{(\log x)^{\ell}}{x^{\sigma}}dx \leqslant (\ell+1)\cdot (\log N^{\sigma -1})^{\ell}\cdot N^{1-\sigma} \cdot \frac{\ell !}{(\sigma -1)^{\ell+1}}\,.
\end{align}

Now fix a positive constant $A$ (only depending on $\ell$) such that $(\ell+1)A^{\ell} \cdot e^{-A} <   1/12$ and let $\sigma - 1 = A/\log N.$ Combining with $(\ref{lower: sigma = 1})$ gives that
\begin{align}\label{large: sigma >1}
   |\zeta^{(\ell)}(s)|  > \frac{\ell !}{(\sigma -1)^{\ell+1}}\cdot (\frac{1}{2} - 2 \cdot \frac{1}{12}) > \frac{\ell !}{3}\cdot \frac{(\log N)^{\ell +1}}{A^{\ell +1}} \gg (\log\log t)^{\ell +1}\,.
\end{align}

Next, define \[f(s):\, = \frac{\zeta^{(\ell)}(s)}{(\log\log s)^{\ell+1}}.\] Suppose that $\zeta^{(\ell)}(1 + it) \neq \Omega( (\log\log t)^{\ell +1})$. So $f(1+it) = o(1)$.
Clearly, $f(2+it) = o(1)$. Then we get a contradiction with \eqref{large: sigma >1} by the Phragm\'en--Lindel\"of principle.

\end{proof}

\section{Discussions, Open Problems and Conjectures}
Let $\ell \in \N$ and $\sigma \leqslant 1$, define the following normalized log-type GCD sums as:
\begin{equation*}  
\Gamma_{\sigma}^{(\ell)}(N):\,= \max_{|\M| = N} \frac{1}{N}\sum_{m, n\in \M} \frac{(m,n)^{\sigma}}{[m,n]^{\sigma}}\log^{\ell} \Big(\frac{m}{(m,n)}\Big)\log^{\ell}\Big(\frac{n}{(m,n)}\Big)\,.
\end{equation*}
\begin{problem}
Given $\sigma \in (0, 1]$ and $\ell \in \N$, optimize $~\Gamma_{\sigma}^{(\ell)}(N)$.
\end{problem}

\begin{remark}
We are particularly interested in the case $\sigma = 1.$ We think  that one can  find constants $C_{\ell}$ such that $~\Gamma_{1}^{(\ell)}(N) \leqslant C_\ell \,(\log_2 N)^{2\ell + 2}$ . If this is the case, find the asymptotically sharp  constant. When $\sigma \in (0, \frac{1}{2})$, is it true that $ N^{1-2\sigma} (\log N)^{2\ell + \alpha(\sigma) } \ll \Gamma_{\sigma}^{(\ell)}(N) \ll N^{1-2\sigma} (\log N)^{2\ell + \beta(\sigma) }$ for some positive constants $\alpha(\sigma), \beta(\sigma)$? (These bounds are inspired by the work of Bondarenko-Hilberdink-Seip in \cite{BHS}, where the authors studied GCD sums for $\sigma \in (0, \frac{1}{2})$ ).
\end{remark}

We are also interested in extreme values of $|\zeta^{(\ell)}(\sigma\!+\! it)|$ in the left half strip. It is unlike the situation of the zeta function, where the values on the left half strip can be easily determined by the right half strip via the functional equation. Thus it worth to study $\Gamma_{\sigma}^{(\ell)}(N)$  when $\sigma < \frac{1}{2}$, even for this reason.

\begin{problem}
Study extreme values of $\,|\zeta^{(\ell)}(\sigma\!+\! it)|$, when $\sigma \in (-\infty , \, \frac{1}{2})$ and $ \ell \in \N$  are given. 
\end{problem}

We can use Theorem \textbf{A} of Ingham \cite{Ingham} to prove  the following claim, from which we obtain the lower bounds \eqref{eq:sigma < 1/2} on maximum of $\,|\zeta^{(\ell)}(\sigma\!+\! it)|$. But we expect something slightly better.
\begin{claim}\label{moment: sigma<1/2}
 Let $\ell \in \mathbb{N}$ and $\sigma \in (-\infty , \, \frac{1}{2})$ be fixed. Then 
\begin{equation*}
    \int_0^T \big|\zeta^{(\ell)}(\sigma\!+\! it)\big| ^2 dt  \sim (2\pi)^{2\sigma-1}\frac{\zeta(2-2\sigma)}{2-2\sigma} T^{2-2\sigma} (\log \frac{T}{2\pi})^{2\ell}.
\end{equation*}
\end{claim}

\begin{proof}
In Theorem  \textbf{A} of Ingham \cite{Ingham}, let $\mu = \nu = \ell$ and $a = b = \sigma$, then
\[
\int_0^T \big|\zeta^{(\ell)}(\sigma\!+\! it)\big| ^2 dt = 2\pi F_{2\ell}(\frac{T}{2\pi}, 2\sigma) + R_{\ell}(T, \sigma),\]

where 
\[
R_{\ell}(T, \sigma) = O(T^{\max \{ 1-\sigma,\, 1-2\sigma\}}(\log T)^{2\ell+2}) = o(T^{2-2\sigma})
\]
and 
\begin{align*}
F_{2\ell}(T, 2\sigma) &= \int_1^T \frac{\partial^{2\ell}}{\partial s^{2\ell}}\Big(\zeta(s)+ x^{1-s}\zeta(2-s) \Big)\Big |_{s = 2\sigma} dx \\
&\sim \zeta(2-2\sigma)\int_{1}^T x^{1-2\sigma} (\log x)^{2\ell}dx \\
&\sim \frac{\zeta(2-2\sigma)}{2-2\sigma} T^{2-2\sigma} (\log T)^{2\ell}.
\end{align*}

\end{proof}

Immediately, we obtain
\begin{corollary} Let $\ell \in \mathbb{N}$, $\beta \in [0, 1)$ and $\sigma \in (-\infty , \, \frac{1}{2})$ be fixed.  Then for large $T$, 
\begin{equation}\label{eq:sigma < 1/2} \max_{T^{\beta}\leqslant t\leqslant T}\left|\zeta^{(\ell)}\Big(\sigma+it\Big)\right| \geqslant (1+o(1)) (2\pi)^{\sigma-\frac{1}{2}}\sqrt{\frac{\zeta(2-2\sigma)}{2-2\sigma}} T^{\frac{1}{2}-\sigma}(\log T)^{\ell}\,. \end{equation}
\end{corollary}

Note that the lower bound in Theorem \ref{Main: sigma =  1} increases when $\ell$ increases. So it's natural to have the following conjecture.
\begin{conjecture}
If $T$ is sufficiently large, then uniformly for all positive integers $\ell_1, \ell_2 $ $\leqslant$ $(\log T)$ $\cdot (\log_2 T)^{-1}$, such that $ \ell_1 < \ell_2$, we have
\begin{align*}
\max_{T \leqslant t\leqslant 2T} \left|\zeta^{(\ell_1)}(1 + it)\right|   < \max_{T \leqslant t\leqslant 2T} \left|\zeta^{(\ell_2)}(1 + it)\right|\,.
\end{align*}
\end{conjecture}

 When $\ell$ is fixed,  we have the following conjecture, inspired by the conjecture of Granville-Soundararajan .
\begin{conjecture}
Let $\ell \in \mathbb N$ be given. Then there exists a polynomial $P_{\ell+1}(x,y)$ of total degree $\ell+1$ such that 
\begin{equation*}
  \max_{T \leqslant t\leqslant 2T}  \left|\zeta^{(\ell)}(1 + it)\right| = P_{\ell+1}(\log_2 T,\, \log_3 T) + o(1), \,\quad as \,\, T \to \infty.
\end{equation*}
In particular, there exists a positive constant $c_{\ell}$ such that
\[
 \max_{T \leqslant t\leqslant 2T}  \left|\zeta^{(\ell)}(1 + it)\right| \sim c_{\ell} \cdot (\log_2 T)^{\ell+1}, \quad as \,\, T \to \infty.
\]
\end{conjecture}

\begin{remark}
  Does  $~\lim_{\ell \to \infty}c_{\ell}$ exist?  In particular, do we have $\lim_{\ell \to \infty}c_{\ell} = 0$?
\end{remark}

When $\ell \in \mathbb N$ and $\sigma \in (0,1)$ are given, we think the maximum of derivatives of zeta function and maximum of zeta function only differs by multiplying some small factors. More precisely, we have the conjecture:
\begin{conjecture}Let $\ell \in \mathbb N$ and $\sigma \in (0,1)$ be fixed, then there exists constants $C(\sigma, \ell)$ and $c(\sigma, \ell)$ which  depend  on $\sigma$ and $\ell$, such that for sufficiently large $T$, we have
\begin{equation*}   (\log T)^{c(\sigma, \ell)} \cdot \max_{ T\leqslant t\leqslant 2T}\left|\zeta(\sigma+it)\right|\ll\max_{ T\leqslant t\leqslant 2T}\left|\zeta^{(\ell)}(\sigma+it)\right| \ll (\log T)^{C(\sigma, \ell)} \cdot  \max_{ T\leqslant t\leqslant 2T}\left|\zeta(\sigma+it)\right|, \end{equation*}
where the implied constants depend at most on $\sigma$ and $\ell$. Moreover, when $\sigma \in (0, \frac{1}{2}]$, then we can take $C(\sigma, \ell) = \ell + \alpha(\sigma) $  and $c(\sigma, \ell) = \ell+ \beta(\sigma)$, where $ \alpha(\sigma) $ and $\beta(\sigma)$ are constants depending at most  on $\sigma.$ 
\end{conjecture}

When we try to give a different proof of Theorem \ref{Main: sigma =  1} via    Levinson's  approach \cite{L}, we meet with the following   problem. In particular, if the following problem has a positive solution, then a new proof for our  Theorem \ref{Main: sigma =  1} can be given. 

\begin{problem}

Let $\ell \in \mathbb N$ be given.
Find $n=n(k)$ and some positive constant $c_{\ell}$ such that if $k$ is sufficiently large, then we have
\begin{equation*}
   \Big( \frac{d_{k,\ell}(n)}{n} \Big)^{\frac{1}{k}}\geqslant c_{\ell} \cdot  (\log k)^{\ell +1} + O((\log k)^{\ell})  
\end{equation*}

and
\begin{equation*}
\log n = k\log k + O(k)  \,, 
\end{equation*}

where $d_{k, \ell}(n)$ is defined as

\begin{equation*}
   d_{k,\ell}(n):\,= \sum_{m_1 m_2 \cdots m_k = n} (\log m_1)^{\ell}(\log m_2)^{\ell} \cdots (\log m_k)^{\ell}\,.
\end{equation*}

\end{problem}

\begin{remark}
The arithmetic function $d_{k,\ell}(n)$ is not multiplicative, which makes the problem difficult. 
 \end{remark}

\begin{problem}
Study extreme values of derivatives of L-functions.
\end{problem}

\begin{problem}
 In our Theorem \ref{Main: sigma =  1}, we require $\ell \leqslant (\log T)(\log_2 T)^{-1}$. What is the largest possible range for $\ell$, that the result of Theorem \ref{Main: sigma =  1} can still be valid. For instance,  what can we say about the extreme values if $\ell = [T],\,  $ or $\ell = [2^T]$? 
\end{problem}

\begin{problem}
  Can one find some range for $\ell $, such that the results in Theorem \ref{Main} can still hold?
\end{problem}

\begin{remark}
The main terms always satisfy since we have $a_{\ell}(n) \geqslant 1 $ for all $n$ and $\ell$. It is not clear about the moments of derivatives of the zeta function if $\ell$ can depend on $T$. For instance, if we let $\ell = [(\log T)(\log_2 T)^{-1}]$, then what can we say about the second moments as $T \to \infty,$
\[\int_0^T \big|\zeta^{(\ell)}(\frac{1}{2}\!+\! it)\big| ^2 dt \sim \quad ?\] When $\ell$ depends on $T$, it also seems difficult to bound the contributions of $~\Delta^{+}+\Delta^{-}$ . 
\end{remark}

Moreover, we have the following general problem, which asks how large or how small the  extreme values of $\,|\zeta^{(\ell)}(\sigma\!+\! it)|$ can be if $\ell$ can be taken arbitrary large  with respect to the length $T$ of the interval $[T, 2T]$. 


\begin{problem}
Given  $\sigma_0 \in [0, 1]$, decide which one of following four properties can be true.

\begin{property*}[A]
Given any function $V:\, (0, +\infty) \to (0, +\infty)$, there always exists some function $f_V:\,(0, +\infty) \to \mathbb N$ such that if $\ell = f_V(T)$, then for sufficiently large $T$, we have
\begin{equation*}
     \max_{T \leqslant t\leqslant 2T}  \left|\zeta^{(\ell)}(\sigma_0 + it)\right| \ll V(T)\,.
\end{equation*}
\end{property*}

\begin{property*}[B]
Given any function $V:\, (0, +\infty) \to (0, +\infty)$, there always exists some function $f_V:\,(0, +\infty) \to \mathbb N$ such that if $\ell = f_V(T)$,  then for sufficiently large $T$, we have
\begin{equation*}
     \max_{T \leqslant t\leqslant 2T}  \left|\zeta^{(\ell)}(\sigma_0 + it)\right| \gg V(T) \,.
\end{equation*}
\end{property*}

\begin{property*}[C]
There exists some function $V:\, (0, +\infty) \to (0, +\infty)$, such that for all function  $f:\,(0, +\infty) \to \mathbb N$, if $\ell = f_V(T)$,  then for sufficiently large $T$, we have
\begin{equation*}
     \max_{T \leqslant t\leqslant 2T}  \left|\zeta^{(\ell)}(\sigma_0 + it)\right| \ll V(T) \,.
\end{equation*}
\end{property*}

\begin{property*}[D]
There exists some function $V:\, (0, +\infty) \to (0, +\infty)$, such that for all function  $f:\,(0, +\infty) \to \mathbb N$, if $\ell = f_V(T)$,  then for sufficiently large $T$, we have
\begin{equation*}
     \max_{T \leqslant t\leqslant 2T}  \left|\zeta^{(\ell)}(\sigma_0 + it)\right| \gg V(T) \,.
\end{equation*}
\end{property*}

\end{problem}

\medskip

\section*{Acknowledgements}
I am grateful to Christoph Aistleitner  for his guidance and many helpful discussions. 
I thank Marc Munsch  for a valuable suggestion. The work  was supported by the Austrian Science Fund (FWF), project W1230.

\end{document}